\documentclass[12pt]{amsart}
\usepackage{amsmath,amssymb,amsbsy,amsfonts,latexsym,amsopn,amstext,cite,
                                               amsxtra,euscript,amscd,bm,mathabx}
\usepackage{url}
\usepackage[colorlinks,linkcolor=blue,anchorcolor=blue,citecolor=blue,backref=page]{hyperref}
\usepackage{color}
\usepackage{graphics,epsfig}
\usepackage{graphicx}
\usepackage{float} 
\usepackage[english]{babel}
\usepackage{mathtools}
\usepackage{todonotes}
\usepackage{url}
\usepackage[colorlinks,linkcolor=blue,anchorcolor=blue,citecolor=blue,backref=page]{hyperref}
\DeclareMathAlphabet{\mathmybb}{U}{bbold}{m}{n}

\usepackage[norefs,nocites]{refcheck}
\hypersetup{breaklinks=true}

\usepackage[norefs,nocites]{refcheck}
\usepackage[english]{babel}
\begin{document}

\newtheorem{thm}{Theorem}
\newtheorem{lem}[thm]{Lemma}
\newtheorem{claim}[thm]{Claim}
\newtheorem{cor}[thm]{Corollary}
\newtheorem{prop}[thm]{Proposition} 
\newtheorem{definition}[thm]{Definition}
\newtheorem{rem}[thm]{Remark} 
\newtheorem{question}[thm]{Open Question}
\newtheorem{conj}[thm]{Conjecture}
\newtheorem{prob}{Problem}
\newtheorem{Process}[thm]{Process}
\newtheorem{Computation}[thm]{Computation}
\newtheorem{Fact}[thm]{Fact}
\newtheorem{Observation}[thm]{Observation}

\newtheorem{lemma}[thm]{Lemma}

\newcommand{\GL}{\operatorname{GL}}
\newcommand{\SL}{\operatorname{SL}}
\newcommand{\lcm}{\operatorname{lcm}}
\newcommand{\ord}{\operatorname{ord}}
\newcommand{\Op}{\operatorname{Op}}
\newcommand{\Tr}{\operatorname{Tr}}
\newcommand{\Nm}{\operatorname{Nm}}
\newcommand{\BigSquare}[1]{\raisebox{-0.5ex}{\scalebox{2}{$\square$}}_{#1}}
\newcommand{\BigDiamond}[1]{\raisebox{-0.5ex}{\scalebox{2}{$\diamond$}}_{#1}}
\numberwithin{equation}{section}
\numberwithin{thm}{section}
\numberwithin{table}{section}

\numberwithin{figure}{section}

\def\sssum{\mathop{\sum\!\sum\!\sum}}
\def\ssum{\mathop{\sum\ldots \sum}}
\def\iint{\mathop{\int\ldots \int}}

\def\wt {\mathrm{wt}}
\def\Tr {\mathrm{Tr}}

\def\SrA{\cS_r\(\cA\)}

\def\vol {{\mathrm{vol\,}}}
\def\squareforqed{\hbox{\rlap{$\sqcap$}$\sqcup$}}
\def\qed{\ifmmode\squareforqed\else{\unskip\nobreak\hfil
\penalty50\hskip1em\null\nobreak\hfil\squareforqed
\parfillskip=0pt\finalhyphendemerits=0\endgraf}\fi}

\def \ss{\mathsf{s}} 

\def \balpha{\bm{\alpha}}
\def \bbeta{\bm{\beta}}
\def \bgamma{\bm{\gamma}}
\def \blambda{\bm{\lambda}}
\def \bchi{\bm{\chi}}
\def \bphi{\bm{\varphi}}
\def \bpsi{\bm{\psi}}
\def \bomega{\bm{\omega}}
\def \btheta{\bm{\vartheta}}

\newcommand{\bfxi}{{\boldsymbol{\xi}}}
\newcommand{\bfrho}{{\boldsymbol{\rho}}}

 \def \xbar{\overline x}
  \def \ybar{\overline y}

\def\cA{{\mathcal A}}
\def\cB{{\mathcal B}}
\def\cC{{\mathcal C}}
\def\cD{{\mathcal D}}
\def\cE{{\mathcal E}}
\def\cF{{\mathcal F}}
\def\cG{{\mathcal G}}
\def\cH{{\mathcal H}}
\def\cI{{\mathcal I}}
\def\cJ{{\mathcal J}}
\def\cK{{\mathcal K}}
\def\cL{{\mathcal L}}
\def\cM{{\mathcal M}}
\def\cN{{\mathcal N}}
\def\cO{{\mathcal O}}
\def\cP{{\mathcal P}}
\def\cQ{{\mathcal Q}}
\def\cR{{\mathcal R}}
\def\cS{{\mathcal S}}
\def\cT{{\mathcal T}}
\def\cU{{\mathcal U}}
\def\cV{{\mathcal V}}
\def\cW{{\mathcal W}}
\def\cX{{\mathcal X}}
\def\cY{{\mathcal Y}}
\def\cZ{{\mathcal Z}}
\def\Ker{{\mathrm{Ker}}}

\def\NmQR{N(m;Q,R)}
\def\VmQR{\cV(m;Q,R)}

\def\Xm{\cX_{p,m}}

\def \A {{\mathbb A}}
\def \B {{\mathbb A}}
\def \C {{\mathbb C}}
\def \F {{\mathbb F}}
\def \G {{\mathbb G}}
\def \L {{\mathbb L}}
\def \K {{\mathbb K}}
\def \N {{\mathbb N}}
\def \PP {{\mathbb P}}
\def \Q {{\mathbb Q}}
\def \R {{\mathbb R}}
\def \Z {{\mathbb Z}}
\def \fS{\mathfrak S}
\def \fB{\mathfrak B}

\def\Fq{\F_q}
\def\Fqr{\F_{q^r}} 
\def\ovFq{\overline{\F_q}}
\def\ovFp{\overline{\F_p}}
\def\GL{\operatorname{GL}}
\def\SL{\operatorname{SL}}
\def\PGL{\operatorname{PGL}}
\def\PSL{\operatorname{PSL}}
\def\li{\operatorname{li}}
\def\sym{\operatorname{sym}}

\def\Mob{M{\"o}bius }

\def\fF{\EuScript{F}}
\def\M{\mathsf {M}}
\def\T{\mathsf {T}}

\def\e{{\mathbf{\,e}}}
\def\ep{{\mathbf{\,e}}_p}
\def\eq{{\mathbf{\,e}}_q}

\def\\{\cr}
\def\({\left(}
\def\){\right)}

\def\<{\left(\!\!\left(}
\def\>{\right)\!\!\right)}
\def\fl#1{\left\lfloor#1\right\rfloor}
\def\rf#1{\left\lceil#1\right\rceil}

\def\Tr{{\mathrm{Tr}}}
\def\Nm{{\mathrm{Nm}}}
\def\Im{{\mathrm{Im}}}

\def \oF {\overline \F}

\newcommand{\pfrac}[2]{{\left(\frac{#1}{#2}\right)}}
\newcommand\numberthis{\addtocounter{equation}{1}\tag{\theequation}}

\def \Prob{{\mathrm {}}}
\def\e{\mathbf{e}}
\def\ep{{\mathbf{\,e}}_p}
\def\epp{{\mathbf{\,e}}_{p^2}}
\def\em{{\mathbf{\,e}}_m}

\def\Res{\mathrm{Res}}
\def\Orb{\mathrm{Orb}}

\def\vec#1{\mathbf{#1}}
\def \va{\vec{a}}
\def \vb{\vec{b}}
\def \vh{\vec{h}}
\def \vk{\vec{k}}
\def \vs{\vec{s}}
\def \vu{\vec{u}}
\def \vv{\vec{v}}
\def \vz{\vec{z}}
\def\flp#1{{\left\langle#1\right\rangle}_p}
\def\T {\mathsf {T}}

\def\sfG {\mathsf {G}}
\def\sfK {\mathsf {K}}

\def\mand{\qquad\mbox{and}\qquad}

\title[Gaps between quadratic forms]
{Gaps between quadratic forms}

\author[Siddharth Iyer] {Siddharth Iyer}
\address{School of Mathematics and Statistics, University of New South Wales, Sydney, NSW 2052, Australia}
\email{siddharth.iyer@unsw.edu.au}

\begin{abstract}
Let $\triangle$ denote the integers represented by the quadratic form $x^2+xy+y^2$ and $\BigSquare{2}$ denote the numbers represented as a sum of two squares. For a non-zero integer $a$, let $S(\triangle,\BigSquare{2},a)$ be the set of integers $n$ such that $n \in \triangle$, and $n + a \in \BigSquare{2}$. We conduct a census of $S(\triangle,\BigSquare{2},a)$ in short intervals by showing that there exists a constant $H_{a} > 0$ with
\begin{align*}
\# S(\triangle,\BigSquare{2},a)\cap [x,x+H_{a}\cdot x^{5/6}\cdot \log^{19}x] \geq x^{5/6-\varepsilon}
\end{align*}
for large $x$. To derive this result and its generalization, we utilize a theorem of Tolev (2012) on sums of two squares in arithmetic progressions and analyse the behavior of a multiplicative function found in Blomer, Br{\"u}dern \& Dietmann (2009). Our work extends a classical result of Estermann (1932) and builds upon work of M{\"u}ller (1989).
\end{abstract}


\maketitle

\tableofcontents
\section{Introduction}
For $n \in \mathbb{Z}$, let $r_{2}(n)$ denote the number of representations of $n$ as a sum of two squares. In 1932 Estermann \cite{Estermann} showed that whenever $a$ is a non-zero integer, one obtains the asymptotic formula
\begin{align}
\label{coreasymptotic}
\sum_{n \leq x}r_{2}(n)\cdot r_{2}(n+a) = C_{a}\cdot x + O(x^{\tau + \varepsilon})
\end{align}
for $\tau = 11/12$ and $C_{a} > 0$. Later, with Weil's work on Kloosterman sums, Hooley provided an improvement to \eqref{coreasymptotic} with $\tau = 5/6$ \cite[p.282]{Hooley2}. Using properties of modular groups, Chamizo showed that $\tau = 2/3$ \cite{Chamizo1}, \cite{Chamizo2}. For two sets of integers $W_{1}$ and $W_{2}$ denote
\begin{align*}
S(W_{1},W_{2},a) = \{n: n \in W_{1}, \ n+a\in W_{2}\}.
\end{align*}
Using the notation $\BigSquare{2}$ to denote the numbers which are a sum of two squares, when $a$ is a non-zero integer we may study the set $S(\BigSquare{2},\BigSquare{2},a)$ and survey the number of its elements in intervals. In this direction, Hooley \cite{Hooley} showed that there exists constants $0<C<D$ depending on $a\neq 0$, so that for large $x \in \mathbb{R}$ one has
\begin{align*}
\frac{C\cdot x }{\log(x)}\leq \# S(\BigSquare{2},\BigSquare{2},a)\cap[1,x] \leq \frac{D\cdot x}{ \log(x)}.
\end{align*}
Using
\begin{align}
\label{4repeqn}
r_{2}(n) = 4\cdot\sum_{d|n}\chi_{4}(d),
\end{align}
where $\chi_{4}$ is the primitive character modulo $4$ \cite[Eqn. 33]{Estermann}, we note that for every $\varepsilon > 0$ there exists a constant $C_{\varepsilon}>0$ so that $r_{2}(n) \leq C_{\varepsilon}\cdot n^{\varepsilon}$, thus a quick application of the asymptotic formula \eqref{coreasymptotic} with $\tau = 2/3$ leads us to a lemma that counts elements of $S(\BigSquare{2},\BigSquare{2},a)$ in short intervals.
\begin{lem}
\label{Chamizoconsequence}
Let $\beta$ and $\varepsilon > 0$. There exists a number $N_{\beta,\varepsilon}>0$ so that
\begin{align*}
\#S(\BigSquare{2},\BigSquare{2},a)\cap[x, x+x^{2/3+\beta}] \geq x^{2/3 + \beta - \varepsilon}
\end{align*}
for $x \geq N_{\beta,\varepsilon}$.
\end{lem}
To bound the largest gaps, one can show the following lemma.
\begin{lem}
\label{Largestgapss2square}
There exists a constant $C_{a}>0$ such that there is an element of $S(\BigSquare{2},\BigSquare{2},a)$ in the interval $[x,x+C_{a}\cdot x^{1/2}]$ whenever $x \geq 1$.
\end{lem}
\begin{proof}
If $a$ is odd and $a' = 2^{t}a$ for some natural number $t$, note that $g \in S(\BigSquare{2},\BigSquare{2},a)$ implies $2^{t}\cdot g \in S(\BigSquare{2},\BigSquare{2},2^{t}\cdot a)$, since the product of a sum of two squares is a sum of two squares. Thus it is enough to prove the lemma for odd values of $a$. Now observe that when $s$ is an integer we have
\begin{align*}
\left(s^2 +\left(\frac{a-1}{2}\right)^2\right) +a= s^2 +\left(\frac{a+1}{2}\right)^2.
\end{align*}
\end{proof}
Let $\triangle$ denote the set of integers represented by the quadratic form $x^2+xy+y^2$. We similarly examine the elements of $S(\triangle,\BigSquare{2},a)$ in short intervals.
\begin{cor}
\label{lowerboundtrianglesquare}
There exists a constant $H_{a} > 0$ such that for any $\varepsilon>0$, and all sufficiently large $x \geq N_{a,\varepsilon}$, one has
\begin{align*}
\# S(\triangle,\BigSquare{2},a)\cap [x,x+H_{a}\cdot x^{5/6}\cdot \log^{19}x] \geq x^{5/6-\varepsilon}.
\end{align*}
\end{cor}
Let $n$ be a positive integer, denote $R_{2}(n)$ to be the number of solutions to the equation $n = x^2+xy+y^2$, and $\chi_{3}$ to be the non-trivial character modulo $3$. A key ingredient in the proof of Theorem \ref{lowerboundtrianglesquare} is the following formula.
\begin{lem}
\label{r3equality}
$R_{2}(n) = 6\cdot\sum_{d|n}\chi_{3}(d)$.
\end{lem}
\begin{proof}
The quadratic field $\mathbb{Q}(\sqrt{-3})$ has discriminant of $-3$, and has the ring of integers $\mathbb{Z}[\frac{1+\sqrt{-3}}{2}]$, which is a principal ideal domain. With $\omega = \frac{1+\sqrt{-3}}{2},$ the function $N : \mathbb{Z}[\omega] \rightarrow \mathbb{N}\cup \{0\}$ given by
\begin{align*}
N(a+b\omega) = (a+b\omega)\cdot(a+b\bar{\omega}) = a^2 +ab+b^2
\end{align*}
 is a norm on the ring $\mathbb{Z}[\omega]$. Solving $N(a+b\omega) = 1$, we verify that $\mathbb{Z}[\omega]$ has $6$ units, namely $1,-1,\omega,-\omega,\bar{\omega},-\bar{\omega}$.

Within the ring $\mathbb{Z}[\omega]$ introduce the equivalence relation $c\sim d$ if and only if $d =u c$, where $u$ is a unit. Denote $\mathcal{I}(n)$ to be the number of ideals in $\mathbb{Z}[\omega]$ with norm (ideal norm) of $n$. From \cite[pg. 301]{Muller} we gather that
 \begin{align}
 \label{Idealformulamod3}
 \mathcal{I}(n) = \sum_{d|n}\chi_{3}(d).
 \end{align}
 For an ideal $W \subseteq \mathbb{Z}[\omega]$ denote $P(W) := [\mathbb{Z}[\omega]: W]$, that is the ideal norm of $W$ in $\mathbb{Z}[\omega]$. If we denote $\overline{W}$ to be the set of elements which are conjugate of elements from $W$, then $\overline{W}$ forms an ideal, and by symmetry we obtain that $P(W) = P(\overline{W})$. Furthermore note that $P(\cdot)$ is completely multiplicative on ideals. Now suppose that $W_{1},\ldots,W_{\mathcal{I}(n)}$ are distinct ideals with $P(W_{j}) = n$.
 Since $\mathbb{Z}[\omega]$ is a principal ideal domain, we can choose integers $a_{j}$ and $b_{j}$ so that $W_{j} = (a_{j}+\omega\cdot b_{j})$, with
 \begin{align*}
 a_{j}+\omega\cdot b_{j} \not\sim a_{k}+\omega\cdot b_{k}
 \end{align*}
for $j \neq k$. We establish the ideal equality
\begin{align*}
W_{j}\cdot \overline{W_{j}} = (a_{j}+\omega\cdot b_{j})\cdot (a_{j}+\bar{\omega}\cdot b_{j}) = (a_{j}^2+a_{j}\cdot b_{j}+b_{j}^2).
\end{align*}
Giving us the equation
\begin{align*}
P((a_{j}+\omega\cdot b_{j})\cdot (a_{j}+\bar{\omega}\cdot b_{j})) = P((a_{j}^2+a_{j}\cdot b_{j}+b_{j}^2)).
\end{align*}
Expanding both sides of the above equation gives us an equation comparing integer types.
\begin{align*}
P((a_{j}+\omega\cdot b_{j}))^2 = (a_{j}^2+a_{j}\cdot b_{j}+b_{j}^2)^2,
\end{align*}
so that $a_{j}^2 +a_{j}b_{j}+b_{j}^2 = n$. Hence the elements $a_{j}+\omega\cdot b_{j}$ for $j = 1,\ldots,\mathcal{I}(n)$ form non-similar (as per the equivalence relation defined above) solutions to $N(a+\omega\cdot b) = n$ in the variables $a,b \in \mathbb{Z}$. Considering the $6$ units of $\mathbb{Z}[\omega]$, there are at least $6\cdot \mathcal{I}(n)$ solutions to $N(a+\omega\cdot b) = n$, in the variables $a,b \in \mathbb{Z}$. Therefore the inequality 
\begin{align}
\label{r3gbound}
R_{2}(n) \geq 6\cdot \mathcal{I}(n)
\end{align}
is gathered.

Conversely note that there are $\frac{R_{2}(n)}{6}$ non-similar solutions to $N(a+b\omega) = n$ over the variables $a,b \in \mathbb{Z}$. It is noted that
\begin{align*}
P((a+\omega\cdot b)\cdot(a+\bar{\omega}b)) =P(n) = n^2,
\end{align*}
giving us that $P((a+\omega\cdot b)) = n$, as $P((a+\omega\cdot b)) = P((a+\bar{\omega}\cdot b))$. Since $\mathbb{Z}[\omega]$ is a principal ideal domain, we note that the non-similar solutions to $N(a+\omega\cdot b) = n$ give rise to distinct principal ideals with ideal norm of $n$. Hence
\begin{align}
\label{r3lbound}
\frac{R_{2}(n)}{6} \leq \mathcal{I}(n).
\end{align}
With \eqref{r3lbound}, \eqref{r3gbound}, \eqref{Idealformulamod3} we verify the formula $R_{2}(n) = \sum_{d|n}\chi_{3}(d)$.
\end{proof}
When $\psi$ is a character modulo $k$ ($k>1$), define 
\begin{align*}
F_{\psi}(n) = \sum_{d|n}\psi(d).
\end{align*}
For natural numbers $a,k$ denote $P(a,k)$ to be the unique integer $t|a$ so that if a prime $p$ divides $t$ then $p|k$ and $\text{gcd}(a/t,k) = 1$. A ``local version'' of a theorem by M{\"u}ller \cite[Theorem 1]{Muller} states the following.
\begin{thm}
\label{Mullerthm}
Let $\psi$ and $\rho$ be primitive characters modulo $k > 1$. If $a \geq 1$, then
\begin{align*}
\sum_{n \leq x}F_{\psi}(n)\cdot F_{\rho}(n+a) = M_{\psi, \rho}(a)\cdot x+O(x^{5/6 + \varepsilon}),
\end{align*}
where 
\begin{align*}
M_{\psi, \rho}(a) = C_{\psi,\rho}(a)+\left\{k^{-1}\sum_{b| P(a,k)}b^{-1}\sum_{j=1}^{k}\psi(j)\cdot\rho((a/b) + j) \right\}\cdot C_{\bar{\psi},\bar{\rho}}(a)
\end{align*}
and
\begin{align*}
C_{\psi,\rho}(a) = \frac{L(1,\rho)\cdot L(1,\psi)}{L(2,\rho\cdot \psi)}\sum_{d|a}\psi(d)\cdot \rho(d)\cdot d^{-1}.
\end{align*}
\end{thm}
With this result, we are able to study the set $S(\triangle, \triangle, a)$ over short intervals; however, it does not suffice to deduce even a slightly weaker version of Corollary \ref{lowerboundtrianglesquare}. This is because we are not allowed to use the characters $\chi_{4}$ and $\chi_{3}$ in Theorem \ref{Mullerthm}. Moreover, lifting these characters to modulus $k$ (where $12 \mid k$) would fail to make them primitive.

For a positive integer $q$ let
\begin{align*}
\eta_{a}(q):= \#\{1\leq \alpha,\beta\leq q: \ \alpha^2 + \beta^2 \equiv a \pmod q\}.
\end{align*}
We have the following theorem.
\begin{thm}
\label{mainsumtwosquaresthm}
Let $\psi$ be a real non-trivial character modulo $b \geq 4$, with $b$ even. If $a \neq 0$ is an integer, then we have
\begin{align}
\label{Theasymptoticformula}
\sum_{n \leq x, \ (n,b) = 1}F_{\psi}(n)\cdot F_{\chi_{4}}(n+a)= \beta(\psi,a)\cdot \eta^{*}(\psi,a)\cdot x+O(x^{5/6}\cdot\log^{19}x),
\end{align}
where 
\begin{align*}
\beta(\psi,a) = \sum_{d= 1}^{\infty}\frac{\psi(d)\cdot\eta_{a}(d)}{d^2} > 0,
\end{align*}
and
\begin{align*}
\eta^{*}(\psi,a) = \sum_{j=1, \ \psi(j-a) = 1}^{b}\frac{\pi \cdot \eta_{j}(b)}{2b^2}.
\end{align*}
\end{thm}
Theorem \ref{mainsumtwosquaresthm} implies Corollary \ref{lowerboundtrianglesquare}. With $\chi_{6}$ being the non-trivial real character modulo $6$, using Lemma \ref{r3equality} and equation \eqref{4repeqn}, one has 
\begin{align*}
\sum_{n\leq x, \ (n,6) = 1}R_{2}(n)\cdot r_{2}(n+a) &= 24\cdot\sum_{\substack{n\leq x, \ (n,6)=1\\ n \geq \max\{1,1-a\}}}F_{\chi_{3}}(n)\cdot F_{\chi_{4}}(n+a)\\
&=24\cdot\sum_{\substack{n\leq x, \ (n,6)=1\\ n \geq \max\{1,1-a\}}}F_{\chi_{6}}(n)\cdot F_{\chi_{4}}(n+a).
\end{align*}
Thus
\begin{align*}
\sum_{n\leq x, \ (n,6) = 1}R_{2}(n)r_{2}(n+a) = \text{ }&24\cdot\beta(\chi_{6},a)\cdot \eta^{*}(\psi,a)\cdot x+\\
&+O(x^{5/6}\cdot \log^{19}(x)).
\end{align*}
The quantity $\eta^{*}(\chi_{6},a)$ takes values from the set
\begin{align*}
\{\eta^{*}(\chi_{6},0),\eta^{*}(\chi_{6},1),\ldots,\eta^{*}(\chi_{6},5)\},
\end{align*}
and each such value is positive. Noting that $\beta(\psi,a)>0$, as proved in Theorem \ref{mainsumtwosquaresthm}, we conclude--by differencing the above asymptotic formula at suitable endpoints--that the remainder of the proof of Corollary \ref{lowerboundtrianglesquare} follows as in the deduction of Lemma \ref{Chamizoconsequence}.

It is remarked that when $K$ is a quadratic number field with discriminant $D$, then there exists a primitive character $\chi_{K}$ modulo $D$ so that the number of ideals of norm of $n$ is given by the expression $F_{\chi_{K}}(n)$ \cite{Muller}. Denote
\begin{align*}
\BigDiamond{K}:&= \{n: \ n\in \mathbb{N}, \ \text{ there exists an ideal }I\subseteq K\text{ with }[\mathcal{O}_{K}:I]=n\}\\
&= \{n: \ n\in \mathbb{N}, \ F_{\chi_{K}}(n)>0\},
\end{align*}
so that with Theorem \ref{mainsumtwosquaresthm}, a result similar to Corollary \ref{lowerboundtrianglesquare} can be obtained in the context of studying the set $S(\BigDiamond{K},\BigSquare{2},a)$. Related to this work, we note deep results by Iwaniec \cite[Thm. 1]{Iwaniec} examining $S(Q,\mathbb{P},a)$, when $Q$ is a set containing numbers that assume a quadratic form, and $\mathbb{P}$ is the set of primes.

Analogous to Lemma \ref{Largestgapss2square}, which attempts to study the largest gaps between elements of $S(\BigSquare{2},\BigSquare{2},a)$, we obtain a weaker gap result for sets of the form $S(\triangle, \BigSquare{2},a)$.
\begin{thm}
\label{Largestgapstrianglesquare}
Let $x \geq 1$. There exists a constant $D_{a}>0$ such that there is an element of $S(\triangle, \BigSquare{2},a)$ in the interval $[x, x + D_{a}\cdot x^{\Upsilon(a)}]$, where
\[
\Upsilon(a) = \begin{cases}
                1/2 & a \in \{n^2-3m^2: \ n,m \in \mathbb{Z}\}\\
                5/8 & a \in \mathbb{Z}\setminus\{n^2-3m^2: \ n,m \in \mathbb{Z}\}.
            \end{cases}
\]
\end{thm}
A generalisation of Theorem \ref{Largestgapstrianglesquare} can be obtained, but we do not pursue that here.
\section{Notation}
\begin{itemize}
\item For two integers $u,v$, we let $(u,v)$ denote the greatest common divisor of $u$ and $v$.
\item For a prime $p$ and an integer $w$ we let $\nu_{p}(w)$ to be the largest integer $t$ with $p^{t}|w$ (defined to be $\infty$ when $w = 0$).
\item $\phi(\cdot)$ is the Euler totient function.
\item $\zeta(\cdot)$ is the Riemann zeta function.
\item Let $U$, $V$, and $W$ be complex quantities dependent on parameters $\mathbf{P}$ from some set, along with (possibly) $d,q \in \mathbb{N}$ and $x \in \mathbb{R}$. Unless otherwise specified, we write $U \ll V$, $V \gg U$ or $U = O(V)$ if there exists an effective constant $C > 0$, independent of $d,q$ and $x$, such that $|U| \leq C\cdot V$ for all choices of $d,q$ and $x$ (with other variables fixed). Further, we write $U = V+O(W)$ if $U-V \ll W$.
\end{itemize}
\section{Proof of Theorem \ref{mainsumtwosquaresthm}}
We interpret a result of Tolev \cite[Eqn. (10)]{Tolev}.
Put
\begin{align*}
S_{q,a}(x) &:= \sum_{\substack{n \leq x \\ n \equiv a \pmod{q}}} F_{\chi_{4}}(n) \\\\
&=\pi\frac{\eta_{a}(q)}{4q^2}x + R_{q,a}(x).
\end{align*}
Then
\begin{align}
\label{TolevLem}
R_{q,a}(x) \ll (q^{1/2} + x^{1/3})(a,q)^{1/2} \tau^4(q) \log^4x.
\end{align}
With the notation $F_{\psi}(\cdot) = A(\cdot)$ and $F_{\chi_{4}}(\cdot) = B(\cdot)$, the main expression we wish to asymptotically evaluate is
\begin{align*}
J(x) := \sum_{n \leq x, \ (n,b) = 1} A(n)\cdot B(n+a).
\end{align*}
We need two lemmas.
\begin{lem}
\label{Avanishing}
If $\psi(n) = -1$ then $A(n)=0$.
\end{lem}
\begin{proof}
Observe that $A$ is the convolution of $\psi$ and $1$, and thus is multiplicative. On primes $p$, with $\psi(p) = -1$ we have that $A_{b}(p^j) = 0$ if $j$ is an odd number. Hence the lemma follows.
\end{proof}
Extend the character $\psi$ to the real numbers by setting $\psi(w) = 0$ whenever $w \not \in \mathbb{Z}$. When $\psi(n) = 1$ and $d|n$, we observe that $\psi(n/d) = \psi(d)$, and this symmetry leads to the following lemma.
\begin{lem}
\label{squareroottrick}
If $\psi(n)= 1$ then
\begin{align*}
A(n) = \left(2\sum_{d < \sqrt{n}}\psi(d)\right) +\psi(\sqrt{n}).
\end{align*}
\end{lem}
With Lemmas \ref{Avanishing} and \ref{squareroottrick} we write
\begin{align*}
J(x) &= \sum_{n\leq x, \ \psi(n) = 1}A(n)\cdot B(n+a)\\
&= \sum_{n\leq x, \ \psi(n) = 1}\left(\sum_{d|n}\psi(d)\right)\cdot B(n+a)\\
&= \sum_{n\leq x, \ \psi(n) = 1}\left(2\cdot\left( \sum_{d|n, \ d<\sqrt{n}}\psi(d)\right) + \psi(\sqrt{n})\right)\cdot B(n+a)\\
&= \sum_{n\leq x, \ \psi(n) = 1}\left(2\cdot \sum_{d|n, \ d<\sqrt{n}}\psi(d) \right)B(n+a) + O(x^{1/2 + \varepsilon}).
\end{align*}
The above equation implies 
\begin{align*}
J(x) = \sum_{1 \leq d < \sqrt{x}}2\psi(d)\left(\sum_{\substack{d^2 < n \leq x \\ \ d|n, \ \psi(n) = 1}}\hspace{-0.5cm}B(n+a)\right) + O(x^{1/2 + \varepsilon}).
\end{align*}
For $(b,d) = 1$, let $V(d)$ be the set of residue classes modulo $bd$ so that if $r$ is a congruence class modulo $bd$, then $r \equiv a \pmod d$ and $\psi(r-a) = 1$. Note that $V(d)$ has $\phi(b)/2$ residue classes. List the residue classes as $r(1,d), r(2,d),...,r(\phi(b)/2,d)$ respectively. Note that
\begin{align*}
J(x) = \sum_{1 \leq d < \sqrt{x}}2\psi(d) \sum_{j=1}^{\phi(b)/2} \left(\sum_{\substack{m \equiv r(j,d)  \pmod{bd}\\ d^2+a<m \leq x+a}}\hspace{-0.5cm} B(m)\right) + O(x^{1/2 + \varepsilon}).
\end{align*}
Let 
\begin{align*}
Q(j;x) = \sum_{\substack{m \equiv r(j,d)  \pmod{bd} \\ d^2+a<m \leq x+a}} \hspace{-0.5cm}B(m).
\end{align*}
By the bound \eqref{TolevLem} we have
\begin{align*}
Q(j;x) =&\pi \frac{\eta_{r(j,d)}(bd)}{4b^2 d^2}(x-d^2) +\\
&+O\left(\left((bd)^{1/2} + x^{1/3}\right)\cdot(r(j,d),bd)^{1/2}\cdot\tau^{4}(bd)\cdot\log^{4}x\right).
\end{align*}
Since $a \neq 0$ we have $(r(j,d),bd) \leq |ab|$,  thus we have
\begin{align*}
Q(j;x) = \pi \frac{\eta_{r(j,d)}(bd)}{4b^2 d^2}(x-d^2) + O(x^{1/3}\cdot\tau^4(d)\cdot\log^4x),
\end{align*}
where the implied constant in the $O-$notation is independent of $d$ in the range $1 \leq d < \sqrt{x}$. Using $\sum_{1\leq d < \sqrt{x}}\tau^4(d) \ll \sqrt{x}\cdot\log^{15}x$ \cite{Wilson} (also see \cite{Luca}), we gather that
\begin{align*}
J(x) &= \left(\sum_{1 \leq d < \sqrt{x}}2\psi(d)\sum_{j=1}^{\phi(b)/2}\pi \frac{\eta_{r(j,d)}(bd)}{4b^2 d^2}(x-d^2)\right) + O(x^{5/6}\cdot\log^{19}x)\\
&= \frac{\pi}{2 b^2}\cdot\left(\sum_{1 \leq d < \sqrt{x}, \ (d,b)=1}\psi(d)\sum_{j=1}^{\phi(b)/2} \frac{\eta_{r(j,d)}(bd)}{d^2}(x-d^2)\right) + O(x^{5/6}\cdot\log^{19}x).
\end{align*}
Noting the $\eta_{s}(\cdot)$ is multiplicative for each $s \in \mathbb{Z}$, we write
\begin{align*}
J(x) =\frac{\pi}{2 b^2}\cdot\left(\sum_{1 \leq d < \sqrt{x}}\psi(d)\sum_{j=1, \ \psi(j-a) = 1}^{b} \frac{\eta_{j}(b)\eta_{a}(d)}{d^2}(x-d^2)\right) + O(x^{5/6}\cdot\log^{19}x).
\end{align*}
Thus we have
\begin{align}
\label{Jbaequation}
J(x) = \frac{\pi}{2b^2}\cdot\left(\sum_{j=1, \ \psi(j-a) = 1}^{b}\eta_{j}(b)\sum_{1 \leq d < \sqrt{x}}\frac{\psi(d)\eta_{a}(d)}{d^2}(x-d^2)\right) + O(x^{5/6}\cdot\log^{19}(x)).
\end{align}
\subsection{Behaviour of $\eta_{a}(d)$} \text{ }
\text{ }

Put $\lambda_{a}(d):= \eta_{a}(d)/d$. With formulas (2.22), (2.23) and (2.24) of \cite{Blomer}, we record the following two lemmas.
\begin{lem}
\label{1mod4adv}
If $p \equiv 1 \pmod 4$ is a prime such that $\nu_{p}(a)> 0$ then
\[
\lambda_{a}(p^{j}) = 
\begin{cases}
1 + j(1-\frac{1}{p}) & 1 \leq j \leq \nu_{p}(a)\\
(1+\nu_{p}(a))(1-\frac{1}{p}) & j \geq \nu_{p}(a)+1.
\end{cases}
\]
\end{lem}

\begin{lem}
\label{3mod4adv}
If $p \equiv 3 \pmod 4$ is a prime such that $\nu_{p}(a)> 0$ then
\[
\lambda_{a}(p^{j}) = 
\begin{cases}
1/p & 1 \leq j \leq \nu_{p}(a), \ j \text{odd}\\
1 & 1 \leq j \leq \nu_{p}(a), \ j \text{ even}\\
1+1/p & j \geq \nu_{p}(a)+1, \ \nu_{p}(a) \text{even}\\
0 & j \geq \nu_{p}(a)+1, \ \nu_{p}(a) \text{ odd}.
\end{cases}
\]
\end{lem}
Furthermore, using (2.20) and (2.21) of \cite{Blomer} we have the following lemma.
\begin{lem}
\label{coprimelem}
Let $p$ be an odd prime such that $p \nmid a$, if $j$ is a positive integer we have
\begin{align*}
\lambda_{a}(p^{j}) = 1 - \chi_{4}(p)/p.
\end{align*}
\end{lem}
Following the proof of \cite[Lemma 2.8]{Blomer}, we also see that when $j \geq 1$ we have
\begin{align}
\label{Twopower}
0\leq\lambda_{a}(2^{j})\leq 4
\end{align}

Let $\overline{\lambda_{a}}$ be the convolution of $\lambda_{a}$ with the M\"obius function $\mu$, that is $\overline{\lambda_{a}} = \lambda_{a}*\mu$. We have that
\begin{align*}
\lambda_{a}(n) = \sum_{d|n}\overline{\lambda_{a}}(d).
\end{align*}
\begin{lem}
\label{lambdaconvbound}
We have
\begin{align*}
\overline{\lambda_{a}}(n) = \frac{f_{\lambda_{a}}(n)}{n},
\end{align*}
where $|f_{\lambda_{a}}(z)| \leq C_{a}$ whenever $z$ is an odd number and $C_{a}>0$ is some positive constant.
\end{lem}
\begin{proof}
Note that $\overline{\lambda_{a}} = \lambda_{a}*\mu$ is a multiplicative function. If $p$ is an odd prime divisor of $a$, then by lemmas \ref{1mod4adv} and \ref{3mod4adv} we have that $\overline{\lambda_{a}}(p^{j}) = \lambda_{a}(p^{j}) - \lambda_{a}(p^{j-1}) = 0$ whenever $j \geq \nu_{p}(a)+2$. Put $D_{a} := \max_{d|a^2}\left|\overline{\lambda}_{a}(d)\right|$. Observe that if $p|a$ then $\nu_{p}(z) \leq \nu_{p}(a) + 1$ or else $\overline{\lambda_{a}}(z) = 0$. If $\overline{\lambda_{a}}(z) \neq 0$ we have
\begin{align*}
\overline{\lambda_{a}}(z) = \overline{\lambda_{a}}(r)\overline{\lambda_{a}}(q),
\end{align*}
where $z = rq$ with $r|a^2$ and $(a,q) = 1$, so that 
\begin{align}
\label{lambdabound}
|\overline{\lambda_{a}}(z)| \leq D_{a}|\overline{\lambda_{a}}(q)|.
\end{align}
If $p$ is a prime so that $(p,a) = 1$ then we compute with Lemma \ref{coprimelem} that $\overline{\lambda_{a}}(p) = \frac{-\chi_{4}(p)}{p}$, and $\overline{\lambda_{a}}(p^{j}) = 0$ for $j \geq 2$. Thus with inequality \eqref{lambdabound} we gather the estimate 
\begin{align*}
|\overline{\lambda_{a}}(z)| \leq D_{a}\frac{1}{q} \leq \frac{a^2D_{a}}{z}.
\end{align*}
Hence $|f_{\lambda_{a}}(z)| \leq a^2 D_{a}$ when $z$ is an odd positive number.
\end{proof}
Recall that $b$ is an even number, for a multiplicative character $\rho_{b}$ modulo $b$, we put
\begin{align*}
J(\rho_{b},b,a,x) = \sum_{n\leq x} \rho_{b}(n)\eta_{a}(n).
\end{align*}
\begin{lem}
\label{Twistedsumlem}
When $\rho_{b}$ is a non-trivial multiplicative character we have
\begin{align*}
J(\rho_{b},b,a,x) \ll x \log(x).
\end{align*}
\end{lem}
\begin{proof}
Let $1\leq \alpha \leq b$ be an integer with $(\alpha,b)=1$. We write 
\begin{align*}
I(\alpha,b,a,x)&:=\sum_{\substack{n \leq x \\ n \equiv \alpha \pmod b}}\hspace{-0.5cm}\eta_{a}(n) = \sum_{\substack{n \leq x \\ n \equiv \alpha \pmod b}}\hspace{-0.5cm}n \lambda_{a}(n)\\\\
&= \sum_{\substack{n \leq x \\ n \equiv \alpha \pmod b}}n \sum_{d|n}\overline{\lambda_{a}}(d).
\end{align*}
Using the notation and result of Lemma \ref{lambdaconvbound} we can continue the above equation as
\begin{align*}
&\sum_{\substack{n \leq x \\ n \equiv \alpha \pmod b}}\left( \sum_{md = n}m f_{\lambda_{a}}(d)\right)= \sum_{d \leq x, \ (d,b) = 1}f_{\lambda_{a}}(d)\sum_{\substack{md \leq x \\ md \equiv \alpha \pmod b}} m\\
&= \sum_{d \leq x, \ (d,b) = 1}f_{\lambda_{a}}(d)\left(\left(\sum_{\substack{md \leq x \\ md \equiv 0 \pmod b}} m\right) + O(x/d)\right)\\
&=\left(\sum_{d \leq x, \ (d,b) = 1}f_{\lambda_{a}}(d)\sum_{\substack{m \leq x/d \\ m \equiv 0 \pmod b}} m \right) + O(x\log(x)).
\end{align*}
Noting that
\begin{align*}
J(\rho_{b},b,a,x) = \sum_{\alpha = 1, \ (\alpha,b) = 1}^{b}\rho_{b}(\alpha)\cdot I(\alpha,b,a,x),
\end{align*}
we have
\begin{align*}
J(\rho_{b},b,a,x) &= \left(\sum_{\alpha = 1, \ (\alpha,b) = 1}^{b}\rho_{b}(\alpha)\right)\left(\sum_{d \leq x, \ (d,b) = 1}f_{\lambda_{a}}(d)\sum_{\substack{m \leq x/d \\ \ m \equiv 0 \pmod b}} m \right)\\
&\hspace{0.5cm}+ O(x\log(x))\\
&= O(x\log(x)).
\end{align*}
\par \vspace{-\baselineskip} \qedhere
\end{proof}
Let 
\begin{align*}
G(\rho_{b},b,a,s) = \sum_{n=1}^{\infty}\frac{\rho_{b}(n)\lambda_{a}(n)}{n^{s}}
\end{align*}
be a complex function in terms of $s$.
\begin{cor}
\label{analytic}
When $\rho_{b}$ is a non-trivial character $G(\rho_{b},b,a,s)$ is analytic on $\text{Re}(s) > 0$.
\end{cor}
\begin{proof}
Let $\delta>0$ be fixed, with Lemma \ref{Twistedsumlem} and summation by parts we have
\begin{align*}
\sum_{n \leq x}\frac{\rho_{b}(n)\lambda_{a}(n)}{n^{\delta}} = \sum_{n \leq x}\rho_{b}(n)\eta_{a}(n)n^{-1-\delta} \ll 1.
\end{align*}
In particular by \cite[Theorem 11.8]{Apostol}, the Dirichlet series $G(\rho_{b},b,a,s)$ converges for $\text{Re}(s)>0$. Using \cite[Theorem 11.12]{Apostol}, we verify that it is also analytic in the said domain.
\end{proof}
We need to show that $G(\rho_{b},b,a,1)$ does not vanish whenever $\rho_{b}$ is a non-trivial character modulo $b$. With the help of Lemmas \ref{1mod4adv} and \ref{3mod4adv} we see that for $s>1$, $G(\rho_{b},b,a,s)$ admits the Euler product of
\begin{align}
\label{Geulerproduct}
G(\rho_{b},b,a,s) = \prod_{p, \ p \text{ is an odd prime}}\left(1+\sum_{d=1}^{\infty}\frac{\rho_{b}(p^{d})\lambda_{a}(p^d)}{p^{ds}}\right).
\end{align}
Put
\begin{align*}
G_{p}(\rho_{b},b,a,s) := 1+\sum_{d=1}^{\infty}\frac{\rho_{b}(p^{d})\lambda_{a}(p^d)}{p^{ds}}.
\end{align*}
\begin{lem}
\label{Gpnonvanishing}
When $p$ is an odd prime and $\rho_{b}$ is a multiplicative character modulo $b$ then $G_{p}(\rho_{b},b,a,1) \neq 0$. Furthermore when $s\geq 1$ and $\rho_{b}$ is a real character then $G_{p}(\rho_{b},b,a,s) > 0$.
\end{lem}
\begin{proof}
Using Lemmas \ref{1mod4adv}, \ref{3mod4adv} and \ref{coprimelem} we bound
\begin{align*}
|G_{p}(\rho_{b},b,a,1)| &> 1 -\sum_{d=1}^{\infty}\frac{1+d(1-\frac{1}{p})}{p^d}\\
&= 1- \frac{1}{p-1} - (1-\frac{1}{p})\frac{p}{(p-1)^2}\\
&= 1-\frac{2}{p-1},
\end{align*}
so that $|G_{p}(\rho_{b},b,a,1)|>0$ for all odd primes. The proof for positivity of $G_{p}(\rho_{b},b,a,s)$ (for real characters) on the region $s\geq 1$ follows similarly.
\end{proof}
Put
\begin{align*}
G_{p}^{*}(\rho_{b},b,a,s) := \sum_{d=2}^{\infty}\frac{\rho_{b}(p^{d})\lambda_{a}(p^d)}{p^{ds}}.
\end{align*}
Using Lemmas \ref{1mod4adv}, \ref{3mod4adv} and \ref{coprimelem} we gather that when $s>1$ we have the estimate
\begin{align}
\label{Gperror1}
|G^{*}_{p}(\rho_{b},b,a,s)| < \sum_{d=2}^{\infty}\frac{d(1-\frac{1}{p})}{p^{ds}} \leq \frac{8}{p^2}.
\end{align}
Let $1_{b}$ be the trivial character modulo $b$. Using Lemmas \ref{1mod4adv}, \ref{3mod4adv} and \ref{coprimelem} we see that $G(1_{b},b,a,s)$ is analytic on $\text{Re}(s) > 1$.
\begin{lem}
\label{zetaestimate}
Whenever $s>1$, We have the estimate
\begin{align*}
0<G(1_{b},b,a,s) \leq K_{a}\cdot\zeta(s),
\end{align*}
for some constant $K_{a}>0$. 
\end{lem}
\begin{proof}
Using Lemmas \ref{1mod4adv}, \ref{3mod4adv}, \ref{coprimelem} and inequality \eqref{Twopower} we see that $\lambda_{a}(\cdot)$ is bounded, by say $K_{a}>0$, so that
\begin{align*}
G(1_{b},b,a,s) \leq K_{a}\cdot\zeta(s).
\end{align*}
Clearly $G(1_{b},b,a,s)$ is positive whenever $s>1$.
\end{proof}
Put
\begin{align*}
H(b,a,s) := \prod_{\rho_{b} \pmod b}G(\rho_{b},b,a,s),
\end{align*}
where the product is taken over all characters modulo $b$. Note that $H(b,a,s)$ is analytic for $\text{Re}(s)>1$. If $G(\rho_{b},b,a,1)$ vanishes for some character $\rho_{b}$, then using Lemma \ref{zetaestimate} and noting that $\zeta$ has a pole of order $1$ at $s = 1$, we see that $H(b,a,s)$ is bounded near $s=1$. With subsequent discussion we will show that this is not the case. Put
\begin{align*}
H_{p}(b,a,s) = \prod_{\rho_{b}\pmod b}G_{p}(\rho_{b},b,a,s).
\end{align*}
Note that when $s>0$ is real then $H_{p}(b,a,s)$ is real valued. We now expand $H_{p}(b,a,s)$ as follows:
\begin{align}
\label{Hpexpansion}
H_{p}(b,a,s) &= \prod_{\rho_{b}\pmod b}\left(1+\frac{\rho_{b}(p)\lambda_{a}(p)}{p^s} +G_{p}^{*}(\rho_{b},b,a,s)\right)\\
&= 1 +\frac{\theta_{b}(p)\lambda_{a}(p)}{p^s} + H^{*}_{p}(b,a,s),
\end{align}
where $\theta_{b}(p) = b$ if $p \equiv 1 \pmod b$ and $0$ otherwise. Using inequality \eqref{Gperror1}, we see that $H^{*}_{p}(b,a,s)$ is a real valued function for $s > 1$, satisfying 
\begin{align}
\label{Hperror}
H^{*}_{p}(b,a,s) \leq \frac{Q_{b}}{p^2}
\end{align}
 for some constant $Q_{b}>0$. Using Lemma \ref{Gpnonvanishing} we see that $H_{p}(b,a,1) \neq 0$. With Lemmas \ref{1mod4adv}, \ref{3mod4adv} and \ref{coprimelem} We can find a positive number $T_{b} \geq 3$ so that $|\frac{\theta_{b}(p)\lambda_{a}(p)}{p^s} + \frac{Q_{b}}{p^2}| \leq \frac{1}{2}$, whenever $p > T_{b}$.
\begin{lem}
\label{loglem}
If $|x| \leq \frac{1}{2}$ we have $\log(1+x) \geq x-2x^2$.
\end{lem}
We now write 
\begin{align*}
H(b,a,s) = \prod_{p \leq T_{b}}H_{p}(\rho_{b},b,a,s) \prod_{p > T_{b}}H_{p}(\rho_{b},b,a,s),
\end{align*}
where with Lemma \ref{Gpnonvanishing} we have
\begin{align*}
\prod_{p \leq T_{b}}H_{p}(\rho_{b},b,a,1) \neq 0.
\end{align*}
For $s > 1$, using equation \eqref{Hpexpansion} we evaluate
\begin{align*}
&\log(\prod_{p > T_{b}}H_{p}(\rho_{b},b,a,s)) = \sum_{p > T_{b}}\log\left(1 +\frac{\theta_{b}(p)\lambda_{a}(p)}{p^s} + H^{*}_{p}(b,a,s)\right)\\
&\geq \sum_{p > T_{b}} \frac{\theta_{b}(p)\lambda_{a}(p)}{p^s}+ H^{*}_{p}(b,a,s)-2\left(\frac{\theta_{b}(p)\lambda_{a}(p)}{p^s} + H^{*}_{p}(b,a,s)\right)^2\\
&\rightarrow \infty \text{ as }s\rightarrow 1^{+}.
\end{align*}
This shows that $\lim_{s \rightarrow 1^{+}}|H(b,a,s)|  = \infty$ so that $H(b,a,s)$ is not bounded near $s = 1$, and that $G(\rho_{b},b,a,1)$ is non-vanishing for all non-trivial characters $\rho_{b}$ modulo $b$. 
Recall
\begin{align}
\label{betaGpsiequality}
\beta(\psi,a) = \sum_{d= 1}^{\infty}\psi(d)\frac{\eta_{a}(d)}{d^2} = \sum_{d= 1}^{\infty}\psi(d)\frac{\lambda_{a}(d)}{d} = G(\psi,b,a,1) ,
\end{align}
so that by the explanation above it is necessary that
\begin{align}
\label{Betapsianeq0}
\beta(\psi,a) \neq 0.
\end{align}
With equation \ref{Geulerproduct}, on the ray $s>1$ we have
\begin{align*}
G(\psi,b,a,s) &= \prod_{p, \ p \text{ is an odd prime}}\left(1+\sum_{d=1}^{\infty}\frac{\psi(p^{d})\lambda_{a}(p^d)}{p^{ds}}\right)\\
&= \prod_{p, \ p \text{ is an odd prime}}G_{p}(\psi,b,a,s),
\end{align*}
and with Lemma \ref{Gpnonvanishing} we have that each entry in the infinite product above is positive. Hence $G(\psi,b,a,s)\geq 0$ when $s > 1$. By Corollary \ref{analytic} we gather that $G(\psi,b,a,1) \geq 0$. By \eqref{betaGpsiequality} and \eqref{Betapsianeq0} it is deduced that $\beta(\psi,a)>0$.

We now use these results in equation (\ref{Jbaequation}) to obtain an asymptotic formula, note that with Lemma \ref{Twistedsumlem} we have
\begin{align*}
J(x) &= \sum_{j=1, \ \psi(j-a) = 1}^{b}\frac{\pi\cdot \eta_{j}(b)}{2b^2}\sum_{1 \leq d < \sqrt{x}}\frac{\psi(d)\eta_{a}(d)}{d^2}(x-d^2) + O(x^{5/6}\cdot\log^{19}x)\\
&= \sum_{j=1, \ \psi(j-a) = 1}^{b}\frac{\pi\cdot \eta_{j}(b)}{2b^2}\sum_{1 \leq d < \sqrt{x}}\frac{\psi(d)\eta_{a}(d)}{d^2}(x)+O(x^{5/6}\cdot\log^{19}x).
\end{align*}
With Lemma \ref{Twistedsumlem} and summation by parts, we note that $\sum_{d \geq \sqrt{x}}\frac{\psi(d)\eta_{a}(d)}{d^2} \ll x^{-1/2+\varepsilon}$. Thus we gather that
\begin{align*}
J(x) = \left(\beta(\psi,a)\sum_{j=1, \ \psi(j-a) = 1}^{b}\frac{\pi \eta_{j}(b)}{2b^2}\right)x + O(x^{5/6}\cdot\log^{19}x),
\end{align*}
which finishes the proof.
\section{Proof of Theorem \ref{Largestgapstrianglesquare}}
Denote $\triangle^{*}$ to be set defined as
\begin{align*}
\triangle^{*}:= \{c^2+3d^2: \ c,d \in \mathbb{Z}\}.
\end{align*}
\begin{lem}
$\triangle^{*} \subseteq \triangle$.
\end{lem}
\begin{proof}
We examine the prime factorization properties of numbers from $\triangle$. By Lemma \ref{r3equality}, note that $n \in \triangle$ if and only if, for every prime $p$ such that $p = 2$ or $p \equiv 5 \mod 6$, one has $\nu_p(n) \equiv 0 \bmod 2$ or $\nu_p(n) = \infty$ (to allow for $n = 0$). It is verified that elements of $\triangle^{*}$ have a similar factorization property. 

If $c,d \in \mathbb{Z}$ with $cd \neq 0$, then
\begin{align*}
\nu_{p}(c^2+3d^2) \equiv \nu_{p}(c_{*}^2 +3d_{*}^2) \bmod 2,
\end{align*}
where $c_{*} = \frac{c}{(c,d)}$ and $d_{*} = \frac{d}{(c,d)}$. Now we record the following:
\begin{itemize}
\item $\nu_{2}(c_{*}^2+3d_{*}^2) \in \{0,2\}$ (via modulo $8$ considerations).
\item If $p \equiv 5 \mod 6$ is a prime, then $\nu_{p}(c_{*}^2+3d_{*}^2) = 0$ (via quadratic reciprocity).
\end{itemize}
Hence the proof is done.
\end{proof}
Suppose that $a  = n^2 - 3m^2$ for some integers $n$ and $m$. Under this assumption, note that for each integer $s$, we have $s^2 +3m^2 \in \triangle$ and $s^2+3m^2+a = s^2+n^2 \in \BigSquare{2}$. Hence Theorem \ref{Largestgapstrianglesquare} holds in this case.

Moving to the more non-trivial situation, we choose integers $l_{1},l_{2}\in \{0,1\}$ so that $l_{1}^2 - 3l_{2}^2 - a \equiv 1 \mod 2$, and require integers $v$ and $d$ to satisfy $v \equiv l_{1} \mod 2$ and $d \equiv l_{2} \mod 2$.  If $(c,d,u,v)$ is an integer solution to 
\begin{align}
\label{trianglesumtwosquareeqn}
c^2 + 3d^2 + a = u^2 + v^2,
\end{align}
then
\begin{align*}
(c+u)(c-u) = v^2-3d^2-a,
\end{align*}
implying that there exists a divisor $g|(v^2-3d^2-a)$ and a simultaneous system of equations satisfying
\begin{align*}
c+u &= g \\
c-u &= \frac{v^2-3d^2-a}{g}.
\end{align*}
Thus the solutions $(c,d,u,v)$ to the equation \eqref{trianglesumtwosquareeqn} come from the set
\begin{align*}
\left\{\left(\frac{1}{2}\left(g+\frac{v^2-3d^2-a}{g}\right), \ d, \ \frac{1}{2}\left(g-\frac{v^2-3d^2-a}{g}\right), \ v\right), \ g,v,d \in \mathbb{Z}, \ g \neq 0\right\}.
\end{align*}
By a trivial substitution, we obtain a parametric family to the equation \eqref{trianglesumtwosquareeqn} given by
\begin{align*}
(c,d,u,v) = \left(\frac{v^2-3d^2-a+1}{2}, d, \frac{-v^2+3d^2+a+1}{2}, v \right).
\end{align*}
It is deduced that 
\begin{align}
\label{parametricfamilytrianglesquarea}
\left(\frac{v^2-3d^2-a+1}{2}\right)^2 +3d^2 \in S(\triangle, \BigSquare{2},a).
\end{align}
Let
\begin{align*}
f(v,d) = \left(\frac{v^2-3d^2-a+1}{2}\right)^2 +3d^2
\end{align*}
be a complex valued function. Define $Q(x)$ to be the smallest integer $d$, $d \equiv l_{2} \mod 2$, such that $f(0,d) > x$. Put $Q^{*}(x) = Q(x)+2$. It is gathered that
\begin{align}
\label{SizeofQtrianglesquare}
x^{1/4}\ll Q^{*}(x) \ll x^{1/4}.
\end{align}
Put
\begin{align*}
E(x) = f(0,Q^{*}(x)) - x.
\end{align*}
By considering the gaps between adjacent values of a quartic polynomial we can show that 
\begin{align}
\label{Errorsizetrianglesquare}
x^{3/4}\ll E(x) \ll x^{3/4}.
\end{align}
So that for $x \gg 1$ we have
\begin{align}
\label{LowerboundDisctriangsquare}
(3\cdot Q^{*}(x)^2 + a - 1)^2-4\cdot E(x) \geq 0.
\end{align}
For a given $x \geq 1$ let
\begin{align*}
g_{x}(w) &:= \left(w+\frac{-3\cdot Q^{*}(x)^2-a+1}{2}\right)^2 +3\cdot Q^{*}(x)^2\\
&= w^2 -\left(-3\cdot Q^{*}(x)^2-a+1\right)w + f(0,Q^{*}(x)),
\end{align*}
so that $g_{x}(\frac{v^2}{2}) = f(v,Q^{*}(x))$. In the variable $w \in \mathbb{C}$ we can solve for the equation
\begin{align*}
g_{x}(w) = x,
\end{align*}
giving us solutions
\begin{align*}
w = \frac{3\cdot Q^{*}(x)^2-a+1 \pm \sqrt{(3\cdot Q^{*}(x)^2-a+1)^2 - 4(f(0,Q^{*}(x))-x)}}{2}.
\end{align*}
For $x$ sufficiently large, so as to satisfy inequality \eqref{LowerboundDisctriangsquare} and $3\cdot Q^{*}(x)^2-a+1 \geq 1$ (the second inequality occurs for large $x$ due to inequality \eqref{SizeofQtrianglesquare}) let
\begin{align*}
w_{*}(x) = \frac{3\cdot Q^{*}(x)^2-a+1 - \sqrt{(3\cdot Q^{*}(x)^2-a+1)^2 - 4\cdot E(x)}}{2}.
\end{align*}
We have 
\begin{align*}
g_{x}(w_{*}(x)) = x,
\end{align*}
and 
\begin{align*}
f\left(\sqrt{2w_{*}(x)},Q^{*}(x)\right) = x.
\end{align*}
We now write
\begin{align*}
w_{*}(x) = \frac{1}{2}\cdot\frac{4\cdot E(x)}{3\cdot Q^{*}(x)^2-a+1+\sqrt{(3\cdot Q^{*}(x)^2-a+1)^2 - 4\cdot E(x)}},
\end{align*}
so that with inequalities \eqref{SizeofQtrianglesquare} and \eqref{Errorsizetrianglesquare} we have 
\begin{align*}
x^{1/4}\ll w_{*}(x) \ll x^{1/4}
\end{align*}
and
\begin{align}
\label{w*xsqrtsandwichinequality}
x^{1/8}\ll\sqrt{2w_{*}(x)} \ll x^{1/8}.
\end{align}
For $x$ large so that $\sqrt{2w_{*}(x)} \geq 3$, let $v_{*}(x)$ be the largest number less than $\sqrt{2w_{*}(x)}$ which satisfies $v_{*}(x) \equiv l_{1} \mod 2$. Note that with inequality \eqref{w*xsqrtsandwichinequality} we have
\begin{align}
\label{v*xsqrtsandwichinequality}
x^{1/8}\ll v_{*}(x) \ll x^{1/8}.
\end{align}
With the element-set relation \eqref{parametricfamilytrianglesquarea} we know that
\begin{align}
\label{elementrelationtrianglesquareapproximation}
\left(\frac{v_{*}(x)^2-3\cdot Q^{*}(x)^2-a+1}{2}\right)^2 +3\cdot Q^{*}(x)^2 \in S(\triangle, \BigSquare{2},a).
\end{align}
It remains to be shown that the quantity mentioned above is ``slightly more'' than $x$. Let $h_{x}(y) := f(y,Q^{*}(x))$ be a function in the variable $y \in \mathbb{R}$.
\begin{lem}
Let $x \gg 1$. In the interval $[0,Q^{*}(x))$, $h_{x}(\cdot)$ is a decreasing function.
\end{lem}
\begin{proof}
Differentiating $h_{x}(y)$ in the variable $y$ we get
\begin{align}
\label{hxderivativeformula}
h_{x}'(y) = 4y\cdot\left(\frac{y^2-3\cdot Q^{*}(x)^2-a+1}{2} \right),
\end{align}
which is less than $0$ when $y \in [0,Q^{*}(x))$, provided $x \gg 1$.
\end{proof}
Let $x \gg 1$, we now evaluate
\begin{align}
\label{trianglesquarepositivity}
\left(\left(\frac{v_{*}(x)^2-3\cdot Q^{*}(x)^2-a+1}{2}\right)^2 +3\cdot Q^{*}(x)^2\right) - x = \hspace{3cm} \notag\\
=h_{x}(v_{*}(x)) - h_{x}\left(\sqrt{2w_{*}(x)}\right)\hspace{2cm} \notag\\
> 0,\hspace{6.5cm}
\end{align}
furthermore with the mean value theorem we evaluate
\begin{align}
\label{trianglesquaremeanvaluethm}
\left|\left(\left(\frac{v_{*}(x)^2-3\cdot Q^{*}(x)^2-a+1}{2}\right)^2 +3\cdot Q^{*}(x)^2\right) - x\right| = \hspace{3cm} \notag\\
= \left|h_{x}(v_{*}(x)) - h_{x}\left(\sqrt{2w_{*}(x)}\right)\right|\hspace{2.5cm}\notag\\
=\left|h_{x}'(z(x))\right|\left(\sqrt{2w_{*}(x)} - v_{*}(x) \right),\hspace{2cm}
\end{align}
where $z(x)$ is a quantity in the interval $\left[v_{*}(x), \sqrt{2w_{*}(x)}\right]$. With \eqref{trianglesquaremeanvaluethm}, \eqref{hxderivativeformula}, \eqref{SizeofQtrianglesquare}, \eqref{v*xsqrtsandwichinequality} and the inequality $\left|v_{*}(x) - \sqrt{2w_{*}(x)}\right| \leq 2$, we obtain 
\begin{align}
\label{trianglesquarefinaleqn}
\left|\left(\left(\frac{v_{*}(x)^2-3\cdot Q^{*}(x)^2-a+1}{2}\right)^2 +3\cdot Q^{*}(x)^2\right) - x\right| \ll x^{5/8}.
\end{align}
With inequalities \eqref{trianglesquarefinaleqn}, \eqref{trianglesquarepositivity}, and the element-set relation \eqref{elementrelationtrianglesquareapproximation}, the theorem follows.
\section{Acknowledgments}
\begin{itemize}
\item The author would like to thank Fernando Chamizo for highlighting difficulties using spectral theory to study the sum \eqref{Theasymptoticformula}. Should their ideas advance to understand interplay between distinct quadratic forms, the main result of this paper should be obsolete. 
\item The author would also like to thank Igor Shparlinski for their comments.
\item This research is made possible due to the Australian Government Research Training Program (RTP) Scholarship and the University of New South Wales for a top-up scholarship, both of which were instrumental in this work.
\end{itemize}
 

\begin{thebibliography}{www}

\bibitem{Apostol} T. M. Apostol, `Introduction to analytic number theory', Undergraduate Texts in Mathematics. New York-Heidelberg: Springer (1976).

\bibitem{Blomer} V. Blomer, J. Br{\"u}dern \& R. Dietmann, `Sums of smooth squares', {\it 	Compos. Math.} {\bf 145(6)} (2009), 1401--1441.

\bibitem{Chamizo1} F. Chamizo, `Correlated sums of r(n)', {\it J. Math. Soc. Japan} {\bf 51(1)} (1999), 237--252.

\bibitem{Chamizo2} F. Chamizo, `The additive problem for the number of representations as a sum of two squares', {\it Mediterr. J. Math.} {\bf 19(1)} (2022), 44.

\bibitem{Estermann} T. Estermann, `An asymptotic formula in the theory of numbers', {\it Proc. London Math. Soc.} {\bf 2(1)} (1932), 280--292.

\bibitem{Hooley2} C. Hooley, `On the intervals between numbers that are sum of two squares', {\it Acta Math.} {\bf 127} (1971), 279--297.

\bibitem{Hooley} C. Hooley, `On the intervals between numbers that are sums of two squares. III.', {\it J. reine angew. Math.} {\bf 267} (1974), 207--218.

\bibitem{Iwaniec} H. Iwaniec, `Primes of the type $\varphi$ (x, y)+ A where $\varphi$ is a quadratic form', {\it Acta Arith.} {\bf 21} (1972), 203--234.

\bibitem{Luca} F. Luca \& L. T{\'o}th, `The rth Moment of the Divisor Function: An Elementary Approach', {\it J. Integer Seq.} {\bf 20(7)} (2017), 3.

\bibitem{Muller} W. M{\"u}ller, `On the asymptotic behaviour of the ideal counting function in quadratic number fields', {\it Monatsh. Math.} {\bf 108(4)} (1989), 301--323.

\bibitem{Tolev} D. I. Tolev, `On the Remainder Term in the Circle Problem in an Arithmetic Progression', {\it Proc. Steklov Inst. Math.} {\bf 276} (2012), 261--274.

\bibitem{Wilson} B. M. Wilson, `Proofs of some formulae enunciated by Ramanujan', {\it Proc. London Math. Soc.} {\bf 21} (1922), 235-255.


\end{thebibliography}
\end{document}